\documentclass{asl}

\usepackage{setspace}
\usepackage{mathrsfs}

\title[Definable choice for a class of weakly o-minimal theories]{Definable choice for a class of weakly o-minimal theories}

\author[Michael~C.\ Laskowski \& Christopher S.\ Shaw]{Michael~C.\ Laskowski \& Christopher S.\ Shaw}

\thanks{A portion of this material was submitted in partial fulfillment of the second author's PhD thesis \cite{Sha} at the University of Maryland under the advisement of the first author. During this time both authors received support from NSF grant DMS-0600217.
Laskowski has also been partially supported by NSF grant DMS-1308546.}


\newtheorem{thm}{Theorem}[section]
\newtheorem{cor}[thm]{Corollary}
\newtheorem{lem}[thm]{Lemma}
\newtheorem{prop}[thm]{Proposition}
\newtheorem{claim}[thm]{Claim}

\theoremstyle{definition}
\newtheorem{defn}[thm]{Definition}
\newtheorem{ex}[thm]{Example}

\newcommand{\M}{{\mathcal M}}
\newcommand{\N}{{\mathcal N}}

\newcommand{\C}{{\mathfrak C}}
\newcommand{\G}{{\mathcal G}}
\newcommand{\ds}{\displaystyle}

\newcommand{\precneq}[1][]{%
{\small \;
  \mathrel{
    \mathop{
      \vcenter{
        \hbox{\oalign{\noalign{\kern-.3ex}\hfil$\prec$\hfil\cr
              \noalign{\kern-.7ex}
              $\neq$\cr\noalign{\kern-.3ex}}}
      }
    }\displaylimits_{#1}
  }
}
\;}

\def\Ind{\setbox0=\hbox{$x$}\kern\wd0\hbox to 0pt{\hss$\mid$\hss} \lower.9\ht0\hbox to 0pt{\hss$\smile$\hss}\kern\wd0}
\def\Notind{\setbox0=\hbox{$x$}\kern\wd0\hbox to 0pt{\mathchardef \nn=12854\hss$\nn$\kern1.4\wd0\hss}\hbox to 0pt{\hss$\mid$\hss}\lower.9\ht0 \hbox to 0pt{\hss$\smile$\hss}\kern\wd0}

\begin{document}

\large

\begin{abstract}
Given an o-minimal structure ${\mathcal M}$ with a group operation, we show that for a properly convex subset $U$, 
the theory of the expanded structure ${\mathcal M}'=({\mathcal M},U)$ has definable Skolem functions precisely when ${\mathcal M}'$ is valuational. As a corollary, we get an elementary proof that the theory of any such ${\mathcal M}'$ does not satisfy definable choice.
\end{abstract}

\maketitle

\section{Introduction}
An \textit{o-minimal structure} is an ordered structure in which all subsets definable with parameters consist of a finite collection of points and open intervals. There is a large body of work in support of the conclusion that o-minimal structures form the tamest possible class of ordered models. Much of this background is developed in the series of papers by J.~Knight, A.~Pillay, and C.~Steinhorn (see \cite{KniPilSte}, \cite{PilSte1}, \cite{PilSte2}); a comprehensive treatment of the subject is also found in the survey text \cite{Van1}. 
Throughout this paper, we will be considering expansions of o-minimal groups.  It is well known that such a group must be both abelian and divisible  (see, e.g., \cite{PilSte1}), hence the underlying order is necessarily dense without endpoints.

More generally, \textit{weakly o-minimal} refers to the class of ordered structures in which each definable subset of the model is a finite union of points and  convex sets, the important distinction being that a 1-dimensional convex set $S$ may be contained in a linear order $\M$ in which the supremum or infimum may not be elements of $M\cup\{\pm\infty\}$. 
The most straightforward way of obtaining a weakly o-minimal structure is to begin with an o-minimal group $(\M,+,<,\dots)$ and let $C\subseteq M$ be any subset which is \textit{downward closed} (for any $c\in C$, if $d<c$ and $d\in\M$, then $d\in C$).  In \cite{BaiPoi}  Y.\ Baisalov and B.\ Poizat prove that the expansion $\M'=(\M,U)$ formed by adding a predicate $U$ whose interpretation is $C$ is weakly o-minimal, and
in fact has a weakly o-minimal theory (i.e., every $\N'$ elementarily equivalent to $\M'$ is also weakly o-minimal).  The motivation for such expansions arose from the work of 
G.\ Cherlin and M.\ Dickmann~\cite{CheDic1}, who obtained a quantifier elimination for the theory RCVF of real closed valued fields.

In addition to the work of \cite{BaiPoi}, D.\ Macpherson, D.\  Marker, and C.\ Steinhorn \cite{MacMarSte} studied the class of weakly o-minimal structures that arose by adding a convex predicate to an o-minimal field.  There, they noted that the analysis of the resulting theory varied greatly depending on whether the cut was \textit{valuational} or \textit{nonvaluational}.
Here, we show that the divide still appears in the more general setting where we expand a group instead of a field.  We exhibit this distinction by considering
the question of whether the resulting theory has definable Skolem functions or definable choice.

\begin{defn}\label{withparam}
A theory $T$  with language $L$ has \textit{definable Skolem functions} if, for every model $\M$ of $T$ and for any $L$-formula $\varphi(\bar{x},y)$ with  lg$(\bar{x})=n$, there is a definable (possibly with parameters) function $F_\varphi: \M^n\longrightarrow \M$, such that  for every $\bar{b}$ from $\M$, if $\M\models \exists y\varphi(\bar{b},y)$, then 
$\M\models\varphi(\bar{b},F_\varphi(\bar{b}))$.
\end{defn}

\begin{defn}\label{defch}A theory $T$ has \textit{definable choice} if, for any $\M\models T$ and for any formula $\varphi(\bar{x},y)$, there is a definable unary function $F$ such that: \begin{enumerate}
\item If $\bar{a}\in M$ and $\M\models\exists y\varphi(\bar{a},y)$, then $\M\models\varphi(\bar{a},F(\bar{a}))$ ($F$ is a Skolem function for $\varphi$); 
\item If $\{b\in M:\M\models \varphi(\bar{a},b)\}=\{b\in M:\M\models \varphi(\bar{a}',b)\}$, then $F(\bar{a})=F(\bar{a}')$.
\end{enumerate}
\end{defn}

Obviously, definable choice implies definable Skolem functions.
It is well known, see e.g., \cite{Van1}, that any o-minimal expansion of a group has both definable Skolem functions and definable choice.
The main results of this paper are summarized in the following theorem.

\begin{thm}  \label{big}  Suppose that $\M$ is an o-minimal group and  $\M'=(\M,U)$ is an expansion by a unary predicate that is interpreted as a downward closed,
proper subset of $C\subseteq M$ for which $\sup(C)\not\in M$.\footnote{If $\sup(C)\in M$, then $\M'$ is o-minimal, so $Th(\M')$ has both definable Skolem functions and definable choice.}  Then:
\begin{enumerate}
\item  If $C$ describes a nonvaluational cut, then $Th(\M')$ does not have definable Skolem functions;
\item  If $C$ describes a valuational cut, then  $Th(\M')$ has definable Skolem functions.
\end{enumerate}
In either case, $Th(\M')$ does not have definable choice.
\end{thm}

It is remarkable that the positive and negative results about Skolem functions arise from very different analyses of the definable subsets in the two cases.
The outline of the paper is as follows.  In Section~\ref{2}, we define valuational and nonvaluational cuts and prove a number of equivalent formulations
in the presence of a group.  In Section~\ref{skolemfail}, we prove Theorem~\ref{big}(1), with the major tool being the analysis of dense pairs of o-minimal structures.
In Section~\ref{tconv}, we prove Theorem~\ref{big}(2).  All of the arguments in this section borrow heavily from work of L.\ van den Dries and A.\ Lewenberg~\cite{VanLew}
where they formed the theory $T_{{\rm convex}}$ to study valuational expansions of o-minimal fields.  Here, however, we show that many of their arguments go through with only a group operation present.  Finally, in Section~\ref{faildc}, we complete the proof of Theorem~\ref{big} by  showing why we never have definable choice in any expansion of this ilk.

\section{Valuational and nonvaluational cuts}  \label{2}

In this section, we define valuational and nonvaluational cuts in ordered structures that are endowed with an ordered group operation.
Our goal is Proposition~\ref{prop.pluslike}, which enumerates a number of equivalents of this notion in the realm of expansions of o-minimal groups.

\begin{defn}A \textit{proper cut} $\langle C,D\rangle$ of a linear order $(M,<)$ is a partition of $M$ into non-empty, disjoint convex sets $C$ and $D$ such that $C<D$.
We say that a proper cut $\langle C,D\rangle$ is \textit{irrational} if  $\sup(C)$ does not exist as  an element of $M$.
\end{defn}

\begin{defn}\label{defval}Given a linearly ordered group $\M$, we say that a cut $\langle C,D\rangle$ of $\M$ is \textit{valuational} if there exists $\varepsilon\in M$ such that $\varepsilon>0$ and $C+\varepsilon = C$. The model $\M$ is \textit{valuational} if there is a definable $C$ such that $\langle C,M\setminus C\rangle$ is a valuational cut, and \textit{nonvaluational} otherwise.
\end{defn}

Note that  a weakly o-minimal structure
$\M$ is valuational if and only if it it has a definable proper nontrivial subgroup: if $\M$ has a valuational definable cut $\langle C,D\rangle$, then the the set $\{\varepsilon: \varepsilon+C=C\}$ is a definable proper nontrivial subgroup; conversely, if $\M$ has a definable proper nontrivial subgroup $G$, then $G$ must be convex (see, \textit{e.g.}, \cite{Sha}), in which case the elements bounded above by $G$ form the downward portion of a valuational cut.

\begin{ex}\label{valstruct}
Fix a nonstandard expansion ${\mathbb R^*}$ of the reals with infinite elements; let $\M_1=\left({\mathbb R}^*,+,<,U\right)$, where $U^{\M_1}$ is the convex hull of ${\mathbb R}$ in ${\mathbb R}^\ast$, and $\M_2=\left(\mathbb{Q},+,<,P\right)$, where $P^{\M_2}=\{x\in{\mathbb Q}:-\pi<x<\pi\}$. Then $\M_1$ is valuational: define $C$, the left portion of a cut, by the formula $U(x)\lor\exists y(x<y\land U(y))$, and use the fact that any two standard reals sum to a standard real. In contrast, since $\M_2$ remains archimedean, it is nonvaluational.
\end{ex}

It is shown in \cite{MacMarSte} that nonvaluational weakly o-minimal fields satisfy monotonicity and cellular decomposition properties which correspond strongly to the o-minimal versions. R.~Wencel extended their results to weakly o-minimal nonvaluational groups (see \cite{Wen2}) but we will not use any of those results here.

Now suppose that $\M$ is an o-minimal structure and $U\subseteq M$ is the downward closed half of an irrational cut of $M$.
Then there is an associated 1-type  generated by the formulas $x>a$ for all $a\in U$ and $x<b$ for all $b\in M\setminus U$.
It is a well known fact that these formulas generate  a complete type $p\in S_1(\M)$, which we dub $tp_\M(\sup U)$, the  \textit{type generated by $U$}.
There are two behaviors of such a type, which we will see precisely correspond to whether or not the cut determined by $U$ is valuational.

\begin{defn}\label{defuniqreal}Let $T$ be o-minimal, $\M\models T$, and $p\in S_1(\M)$ be the complete type generated by an irrational cut of $M$. Then $p$ is \textit{uniquely realizable} if there is an elementary extension $\M'$ of $\M$ which realizes $p$ with precisely one element. Equivalently, given $b\in{\mathfrak C}$ which realizes $p$, $p$ is realized uniquely by $b$ in every prime model over $M\cup\{b\}$. 
\end{defn}

The following result is proved by D.\ Marker in \cite{Mar1}.

\begin{thm}\label{irratcut1}  Let $\M$ be o-minimal, $p\in S_1(\M)$ the complete type generated by an irrational cut.  If $p$ is uniquely realizable, $b\in{\mathfrak C}$ realizes $p$, and $M(b)$ is prime over
$M\cup\{b\}$, 
then every element of $M(b)\setminus M$ realizes a uniquely realizable type over $M$. \end{thm}

\begin{lem}\label{cuts} Let $\M$ be an o-minimal group, $U$ a downward closed subset of $M$ describing an irrational cut of $M$.
 Then $p=tp_\M(\sup U)$ is nonuniquely realizable if and only if the expanded structure $\M'=(\M,U)$ is valuational.
\end{lem}

\begin{proof}
For the forward direction, suppose $a,b\in\C$, with $a<b$ and both $a$ and $b$ realizing $p$ in some prime model over $M\cup\{b\}$. 
Because $\M$ is a divisible ordered abelian group, we may define $\varepsilon=\frac{b-a}{2}$, where $\varepsilon\in M$, and $0<\varepsilon<b-a$. 

Let $c\in U^\M$. Then $c+\varepsilon<a+\varepsilon < b$. Since $c,\varepsilon\in M$, then $c+\varepsilon\in M$; and $tp_{\M}(b)$ is generated by $p$, which implies that $c+\varepsilon\in U^{\M'}$. Thus $\M'$ is valuational.

For the converse, suppose $\M'$ is valuational. Then let $\varepsilon\in M$ with $0<\varepsilon$ be such that for any $a\in U^{\M'}$, $a+\varepsilon\in U^{\M'}$. Note that this means for any $d\in M$, if $d>U^{\M'}$, then $d-\varepsilon>U^{\M'}$ (otherwise $(d-\varepsilon)\in U^{\M'}$ and $(d-\varepsilon)+\varepsilon\not\in U^{\M'}$). Choose a realization $b$ of $p$. 

Observe that $b+\varepsilon\models p$: if not, then there is $d\in M$ such that $d>U^{\M'}$ and $b+\varepsilon>d$, so clearly $b>d-\varepsilon$. But by the above comment, $d-\varepsilon>U^{\M'}$, a contradiction. So $p$ is realized in every prime model over $M\cup\{b\}$ by $b$ and $b+\varepsilon$, and in particular $b$ is not the unique realization.
\end{proof}

\begin{cor}\label{cor.dense}
With $\M$, $U$, and $p$ as above, and $(\M, U)$ nonvaluational, let $b\in\C$ be a realization of $p$, and $\N$ be a prime model over $M\cup\{b\}$. Then $\M$ is dense in $\N$.
\end{cor}

\begin{proof} By Lemma \ref{cuts}, $p$ is uniquely realized in $\N$ by $b$. Thus any irrational cut whose complete type over $M$ is realized in $\N$ is uniquely realized in $\N$ (else there are $d<d'$ such that $tp_\C(d/M)=tp_\C(d'/M)$, whence $d+(b-d)$ and $d'+(b-d)$ both realize $p$). Thus for any $b<b'\in N$, $b$ and $b'$ must realize different cuts over $M$; as such, there is $a\in M$ such that $b<a<b'$.
\end{proof}

We note that there is nothing particularly magic about the operation of addition in defining a valuational cut.  We indicate here a class of  surrogate functions.

\begin{defn}\label{defpluslike} Let $(\M,<)$ be an ordered group, and $F:M^2\rightarrow M$ a definable function. $F$ is \textit{pluslike on $M$} if the following hold:
\begin{itemize}
\item $F$ is continuous on $M^2$;
\item For every $b\in M$, the unary function $F_b(x):=F(x,b)$ is strictly increasing; and
\item For every $a\in M$, the unary function $_a F(y):=F(a,y)$ is strictly increasing.
\end{itemize}
\end{defn}

Note that $+^\M$ itself is pluslike. Pluslike functions can be seen as capturing the topological properties of the additive structure of $\M$ without concern for the group structure. Just as valuational convex sets are defined in terms of $+^\M$, we define an \textit{$F-$valuational} set in terms of a binary pluslike function $F$.

\begin{defn} Let $\M$ be a weakly o-minimal expansion of an ordered group, $\langle C,D\rangle$ a definable cut of $\M$, and $F:\M^2\rightarrow \M$ a definable pluslike function. Then $\langle C,D\rangle$ is \textit{$F$-valuational} if there is $\varepsilon> 0$ from $M$ such that for all $a\in C$, $F(a,\varepsilon)\in C$. A definable downward-closed convex set $U\subseteq M$ is \textit{$F$-valuational} if the cut $\langle U^\M,\{x\in M: x> U\}\rangle$ is $F$-valuational.
\end{defn}

\begin{prop}\label{prop.pluslike} Let $\M$ be an o-minimal expansion of an ordered group, $U$ a predicate for a properly convex downward-closed subset, and $\M'=(\M,U)$. The following are equivalent:
\begin{itemize}
\item[(i)] $U$ is $F$-valuational for some definable pluslike function $F$.
\item[(ii)] $U$ is $F$-valuational for all definable pluslike functions $F$.
\item[(iii)] $U$ is valuational.
\item[(iv)] $p:=tp_\M(\sup U)$ is nonuniquely realizable.
\end{itemize}
\end{prop}

\begin{proof}
(ii)$\Rightarrow$(i) is trivial, and we previously showed (iii)$\Leftrightarrow$(iv); thus it suffices to show (i)$\Rightarrow$(iv)$\Rightarrow$(ii).

(i)$\Rightarrow$(iv) is essentially the same as the forward direction in Lemma \ref{cuts}: suppose $F$ and $\varepsilon$ witness the fact that $U$ is $F$-valuational. First note that since $F$ is pluslike, it is strictly increasing in both variables; thus for a fixed $a\in M$, $_aF$ is a bijection onto some subset of $M$, and has a definable continuous partial inverse $ _aF^{-1}$. Now note that for any $d\in M$, $d>U^\M$ implies that $_dF^{-1}(\varepsilon)>U^\M$ (or else $d>U^\M$ and $_dF^{-1}(\varepsilon)\in U^\M$, so $F(_dF^{-1}(\varepsilon),\varepsilon)=d\in U^\M$, a contradiction).

Suppose $p$ is realized by $b$ in some prime model $\M_0$ over $M\cup\{b\}$. We claim that $F(b,\varepsilon)\models p$ as well. If not, then there is
$d\in M$ such that $d>U^\M$, and $F(b,\varepsilon)>d$. Thus, $b>F^{-1}_d(\varepsilon)$. But $F^{-1}_d(\varepsilon)>U^\M$, contradicting the statement in the above paragraph. So, $b\models p$ and $F(b,\varepsilon)\models p$. And because $F$ is definable, we have $F(b,\varepsilon)\in \M_0$. Since $F$ is strictly increasing, $F(b,\varepsilon)\neq b$, contradicting unique realizability of $p$.

To show (iv)$\Rightarrow$(ii), let $F:M^2\longrightarrow M$ be the pluslike function defined by
$$F(x_1,x_2)=y \text{ if and only if } \varphi(x_1,x_2,y)$$
Suppose $b\models p$, and let $a<b$ such that $a\models p$ and $a$ is an element of a prime model $\M_0$ over $M\cup\{b\}$. That $\varphi$ defines a pluslike function is a first-order property, so let $F'$ be the pluslike function defined by $\varphi^{\M_0}$. Since ${\mathcal M}\preceq \M_0$, we have $\M_0\models \forall y\forall x_1\exists !x_2\left(\varphi(x_1,x_2,y)\right)$. Let $\varepsilon$ be the unique value such that $\varphi(a,\varepsilon,b)$.

If $\varepsilon\in M$, then let $c\in U^\M$ such that $F'(c,\varepsilon)<F'(a,\varepsilon)=b$. Then $c\in M$ and $\varepsilon\in M$ implies $F(c,\varepsilon)\in M$, and thus $F(c,\varepsilon)\in U^\M$, so $U$ is $F$-valuational. And if $\varepsilon\not\in M$, then since $p$ is the type of an irrational cut, then by Theorem \ref{irratcut1}, there is $\varepsilon_0\in M$ such that $0<\varepsilon_0<\varepsilon$. We may then redo the above argument with $\varepsilon_0$.
\end{proof}

\section{Failure of Skolem functions for nonvaluational expansions}\label{skolemfail}

As a first attempt to determine whether or not a weakly o-minimal model has Skolem functions,  we describe a useful characterization that follows easily by compactness. 
Alternatively, it follows as a consequence of van den Dries' `Rigidity condition' i.e., Theorem~2.1 of \cite{Van4}.

\begin{lem}\label{vantest}  Suppose $T$ is a theory in a language with at least one constant symbol that admits quantifier elimination.
Then $T$ has definable Skolem functions if and only if $T$ has a definitional expansion $T^{df}$ that is universally axiomatizable.
\end{lem}

This test cannot be used conveniently in the case of a nonvaluational structure.
Because ``$\M$ is nonvaluational'' is expressed by an $\forall\exists$ axiom, substructures of nonvaluational weakly o-minimal structures are not well-behaved, as illustrated by the following example.

\begin{ex}Let $\M=({\mathbb Q}^3,+,<,U)$, where $<^\M$ is lexicographic, $+^\M$ is componentwise addition in ${\mathbb Q}^3$, and $P^\M=\{\bar{x}\in{\mathbb Q}^3:\bar{x}<(1,1,\pi)\}=\{(n,p,q):n\leq 1,\,p\leq 1,\,q<\pi\}$. That $\M$ is weakly o-minimal is clear, since $({\mathbb Q}^3,+,<)\equiv ({\mathbb Q},+,<)$ and $P^\M$ is convex. $\M$ is nonvaluational, since every `small' element is of the form $(0,0,q)$ for some $q\in {\mathbb Q}$, and ${\mathbb Q}$ is itself archimedean. We note that $\M$ has substructures $\M'$ and $\M''$ such that $\M'$ is valuational, and $\M''$ is o-minimal: Let $\M'$ be the substructure generated by $+^\M$ with universe $\{(n,p,q):p=0\}$. Then $U^{\M'}=\{(n,0,p):n\leq 1\}$, which is valuational in $\M'$, witnessed by $\varepsilon=(0,0,1)$. And let $\M''$ be the substructure with universe $\{(n,p,q): q=0\}$. Then $U^{\M''}=\{(n,p,0):(n,p,0)\leq (1,1,0)\}$, which is an interval with endpoints in $\M''$; thus $\M''$ is o-minimal.
\end{ex}

To show that the obstruction alluded to above is in fact fatal,  we state what can be discerned about definable sets and functions in a nonvaluational expansion.
Corollary~\ref{cor.dense}
allows us to use the van den Dries analysis of definable functions in dense pairs which appears in \cite{Van2}. Formally, a \textit{dense pair} is a pair $(\N,\M)$ such that $\M$ and $\N$ are o-minimal ordered abelian groups with $\M\preceq\N$, and such that $\M$ is dense in $\N$.
We shall make use of the following, stated as Theorem 3, part (3) in \cite{Van2}:

\begin{thm}\label{vdddense}
Let $(\N,\M)$ be a dense pair, and let $f:M^n\rightarrow M$ be definable in $(\N,\M)$. Then there are $f_1,\ldots,f_k:M^n\rightarrow M$ definable in $\M$ such that for each $m\in M^n$, we have $f(m)=f_i(m)$ for some $i\in\{1,\ldots,k\}$.
\end{thm}

Finally, we note that working in a weakly o-minimal structure obtained by adding a nonvaluational convex predicate to an o-minimal structure is essentially the same as working in a dense pair. The following lemma makes this notion explicit:

\begin{lem}\label{nonvaldense}
Let $\M$ be o-minimal, $U$ be a downward-closed nonvaluational properly convex subset, and $\M'=(\M,U)$. Let $\N$ be a prime model over $M\cup\{b\}$, where $b$ realizes $tp_\M(\sup U)$. Then for any $X\subseteq M$ definable in $\M'$, there is a formula $\varphi_X(\bar{x},y)$ such that $X=\varphi_X(\N^n,b)\cap \M^n$.
\end{lem}

\begin{proof}
Given the formula $\psi(x)$ which defines $X$ in $\M'$, replace all instances of $U(x)$ in $\psi$ with $x<y$.
\end{proof}

\begin{thm}\label{nonvalnosk}
Let $\M$ be an o-minimal structure with a group operation, $U$ be a downward-closed convex subset of $M$, and $\M'=(\M,U)$. If $\M'$ has a definable, nonvaluational cut which is not definable in $\M$, then $\M'$ does not have definable Skolem functions.
\end{thm}

\begin{proof}
The proof will analyze $p=tp_\M(\sup U)$. 
We proceed by showing there is
no $\M'$-definable function $f:M\rightarrow M$ such that for every $a\in U^{\M'}$, we have $a<f(a)$ and $f(a)\in U$.

Let $f:M\rightarrow M$ be such a function, and assume $f$ is $\M'$-definable. Consider $\Gamma(f)\subseteq M^2$, the graph of $f$ on $M$.  Let $b\in\C$ be a realization of $p$, and $\N$ be a prime model over $M\cup\{b\}$. Then by Lemma \ref{nonvaldense}, there is an $L$-formula $\varphi(x,y,\bar{z})$, and a tuple $\bar{c}$ from $N$ such that $\{(x,y):\N\models\varphi(x,y,\bar{c})\}\cap M^2=\Gamma(f)$. Note that the solution set of $\varphi(x,y,\bar{c})$ in $\N$ need not be the graph of a function, but the formula defines $\Gamma(f)$ in the pair $(\N,\M)$. Hence, by Theorem \ref{vdddense}, there are $\M$-definable functions $f_1,\ldots, f_k$ such that for each $a\in U$, $f(a)=f_i(a)$ for some $i\leq k$. Each of these functions $f_i$ is also definable in $\M'$, so by weak o-minimality, there is a convex set $I$, which is unbounded in $U$, and $i\leq k$  such that $f\upharpoonright I=f_i$. Restricting the domain of $f$, we may assume $f=f_i$.

The main point of the dense pair argument was to get this function $f$ to be definable in $\M$. Thus by monotonicity for o-minimal structures, we may assume $I$ is an open \textit{interval}. Now $I$ is unbounded in $U^{\M'}$ and $tp_\M(\sup U)$ is omitted in $\M$; thus, $I$ must be $(c,d)$ for some $c,d\in M$, $c\in U^{\M'}$ and $d>U^{\M'}$.

By the hypotheses on $f$, we may shrink the domain further to assume $f$ is strictly increasing on $I$, and for all $a\in U$, $a<f(a)$ and $f(a)\in U$. Let $g(x):=f(x)-x$. Then we may also assume by shrinking that $g$ is strictly decreasing and positive on $I$. 

\pretolerance=100000
Again by the hypotheses on $f$, we can say in a certain sense that $\ds{\lim_{x\rightarrow\sup U} g(x)=0}$: since $\M'$ is nonvaluational, for any $\varepsilon>0$, there is $a\in U^{\M'}$ and $b>U^{\M'}$ such that $b-a\leq\varepsilon$. Then $a<f(a)<b$ implies that $f(a)-a<b-a$, and thus $g(a)=f(a)-a<\varepsilon$.

But since $I$ is an interval with right endpoint $d\in M$, $d>U^{\M'}$, there is $d'\in I\cap U^{\M'}$. Let $\varepsilon_0=f(d')-d'=g(d')$. Then by the above, there is $a'\in U^{\M'}$ such that $g(a')<\varepsilon_0$. $g$ was strictly decreasing on $I$, so we must have $g(d')<\varepsilon_0$, an impossibility.
\end{proof}

The authors note that recent, independent work by Eleftheriou, Hasson, and Keren (see \cite{Ker} and \cite{Elef}), has generalized the above result, showing that if $\M^*$ is any proper expansion of $\M$ by non-valuational cuts, then $\M^*$ does not have definable Skolem functions. 

\section{$T$-resistance and Skolem functions for valuational expansions}\label{tconv}

To explore Skolem functions in the valuational case, we first consider the specialization to \textit{$T$-resistant} structures, an analogue of the $T$-convex structures discussed in \cite{VanLew} and \cite{Van3}. 
In those works, the authors were motivated by a desire to generalize the theory of real-closed valued fields (RCVF), which is a key example of a $T$-convex theory. 
For that reason, the authors restrict their study to that of o-minimal fields; however, here we  only assume that $T$ expands a group. 
Our proof of Theorem \ref{resQE}, below, is is based on the proof of Theorem 3.3 and Corollary 3.13 of \cite{VanLew}.

\begin{defn}\label{defTres}Let $\M$ be an o-minimal expansion of a group with complete theory $T$, and $U$ be a proper, convex subgroup of $M$. We say that the model $(\M,U)$ is \textit{$T$-resistant}, if the following conditions hold:
\begin{enumerate}
\item The language of $\M$ contains at least one constant symbol $c$ interpreted by a non-zero element of $U$. 
\item $U$ is closed under every $0$-definable continuous total function with domain $M$. 
\end{enumerate}
\end{defn}

The following results, as well as the proof structure and the supporting lemmas, are closely patterned after similar results for $T$-convex theories in \cite{VanLew} and \cite{Van3}. 
Suppose the $L$-structure $\G$ is an o-minimal expansion of a group, let $T=Th(\G)$,  and suppose that $U\subseteq\G$ is a proper, $T$-resistant subgroup.  Let 
$L^*=L\cup\{U,c\}$
and let $$T^*=Th(\G)\cup \hbox{`$U$ is $T$-resistant'}\cup \hbox{`$c>U$'}$$

\begin{thm}\label{resQE} 

Suppose $\G$ is an o-minimal expansion of a group, such that $T=Th(\G)$ has quantifier elimination, and is universally axiomatizable. Then the $L^*$-theory $T^*$ defined above
admits elimination of quantifiers and is universally axiomatizable.
\end{thm}

\begin{proof}
As axioms asserting that `$U$ is $T$-resistant' and `$c>U$' are visibly universal, the fact that $T^*$ is universally axiomatizable is automatic.
For the quantifier elimination, our proof follows, with only minor changes,  the proof of Theorem 3.10 from \cite{VanLew}. As in that article, we use a variant on the Robinson-Shoenfield substructure-completeness test for quantifier elimination of $T^*$. We provide a series of lemmas constituting a sketch of the argument. In many cases, the proofs are identical to corresponding items from \cite{VanLew}; in those cases, we refer the reader to the article for further details.

\begin{lem}\label{lem3.4} To show that a model $(\G,U)$ is $T$-resistant, it suffices to show that $U$ is closed under all functions $f$ which:
\begin{itemize}
\item are $0$-definable in $\G$, strictly increasing, continuous, even, 
\item and which satisfy $\ds{\lim_{x\to+\infty} f(x)=+\infty}$. 
\end{itemize}
\end{lem}
This is proved as Lemma 3.4 from \cite{VanLew}, and functions satisfying the conditions of the lemma are referred to as \textit{$T$-functions}. 

\begin{lem}\label{lem3.5} Let $f$ be a $T$-function on $\G$ and $U$ a $T$-resistant subgroup of $\G$. Adopt the notation $|A|=\{|a|:a\in A\}$. Then $f(|U|)=|U|$, and $f(|G\setminus U|)=|G\setminus U|$.
\end{lem}

The proof of Lemma \ref{lem3.5} is identical to that of Lemma 3.5 from \cite{VanLew}. The main lemma which follows is adapted from Lemma 3.6 of \cite{VanLew}. 

\begin{lem}\label{mainlemma}
Let $U$ be a $T$-resistant subgroup of $\G\preceq {\mathcal H}$, and $a\in H$ such that $|U|<a<|G\setminus U|$. There is exactly one $T$-resistant subgroup $U_a$ of $\G\langle a\rangle$ (the prime model of $T$ over $G\cup\{a\}$, which in an o-minimal field corresponds to the definable closure of $a$ over $G$) containing $a$ such that $(\G,U)\subseteq (\G\langle a\rangle,U_a)$, namely:
$$U_a=\left\{x\in\G\langle a\rangle:|x|<|\G\setminus U|\right\}$$
\end{lem}
\textit{Proof of main lemma:} First, we show that $U_a$ satisfies the conclusion of the lemma. Define $U_a$ in terms of $a$ as
\begin{align*}
U_a:= & \{x\in\G\langle a\rangle:  |x|\leq f(\bar{u},a)\text{ for some 0-definable }\\& \text{  continuous function }f:{\mathcal H}^{m+1}\to {\mathcal H},\,\bar{u}\in U^m\}
\end{align*}

That $U_a$ is convex is clear from the definition. To see that $U_a$ is in fact $T$-resistant, choose  $b\in U_a^{>0}$ with $|b|\leq f(\bar{u},a)$, and any $T$-function $g$ on $\G\langle a\rangle$.  Because $g$ is strictly increasing in absolute value, we have $|b|<|g(b)|$; since $g$ is 0-definable, then so is $g\circ f$, giving $|b|<g\circ f(\bar{u},a)$. By (\ref{lem3.4} - \ref{lem3.5}) above, $U_a$ is $T$-resistant.

Note also that $U_a$ is the minimal $T$-resistant subset of $\G\langle a\rangle$ containing $a$, since any $T$-resistant subgroup $U$ of $\G\langle a\rangle$ must be closed under all 0-definable functions $f:U^n\to U$. 

We claim that $(\G,U)$ is a submodel of $(\G\langle a\rangle,U_a)$. This can be reduced to showing that $U_a\cap G = U$. 

First, we know that there is an expansion $(\G^*,U^*)\succeq(\G,U)$ such that $U^*\cap G=U$. To wit, define $p(x)$ to be the type
$$\{x>b:b\in U\}\cup\{x<b:b>U\},$$
and let $\G^*$ be an elementary extension $p$ with an element $a^*\in\C$. If $\G^*=\G\langle a^*\rangle$ and $U^*=U^{\G^*}$, we have $U^*\cap G=U$, and $U<a^*<|G\setminus U|$. Because $a$ and $a^*$ have the same type over $\G$, we may assume $a^*=a$, and $G^*=G\langle a\rangle$. Then $U_a\subseteq U^*$ by minimality of $U_a$, so $U_a\cap G\subseteq U^*\cap G = U$; thus, $U_a\cap \G=U$.

For uniqueness, we rely upon a technical result, which appears as Lemma 2.13 in \cite{VanLew}:

\begin{lem}\label{2.13}
Let $U$ be a $T$-resistant subgroup of $\G$ containing $\G'\preceq \G$, and define $INF_{\G'}:=\{x\in G:|x|<\varepsilon$ for each $\varepsilon$ in $(\G')^{>0}\}$, the set of elements of $G$ which are infinitesimal relative to $G'$.

Then, if  $a\in U$ with $a\not\in \G'+INF_{\G'}$, we have $\G'\langle a\rangle\subseteq U$.
\end{lem}

Given the technical lemma, suppose that there is $U^*\supsetneq U_a$ with $a\in U^*$ and $(\G,U)\subseteq (G\langle a\rangle,U^*)$, \textit{i.e.}, $U^*\cap G=U$. Then, because $U^*$ properly extends $U$, there is $a_1\in U^*\setminus U_a$ with $|U_a|<a_1<|G\setminus U|$; in particular, $|U|<a_1<|G\setminus U|$. Since tp$_{\G}(a)=$tp$_\G(a_1)$, there is an automorphism $h$ of $G\langle a\rangle$ fixing $G$ pointwise such that $h(a)=a_1$. Write $a_1=t(\bar{g},a)$ for some term $t$ of $L^{df}$ (which consists of $L$ together with all 0-definable functions of $\G$), and a tuple $\bar{g}$ from $G$. Applying $h$ successively, define $a_{i+1} = h^{i+1}(a) = t(\bar{g},a_i)$. For each $n$, define:
\begin{align*}
U_n:= & \{x\in\G\langle a\rangle:|x|\leq f(\bar{u},a_n)\text{ for some 0-definable} \\ &\text{ continuous function }f:{\mathcal H}^{m+1}\to {\mathcal H},\,\bar{u}\in U^m\}
\end{align*}

Let ${\mathcal P}\langle \bar{g},a\rangle$ be the prime model of $T$ over the tuple $\bar{g}a$. Then for all $n$, $a_n$ is an element of ${\mathcal P}\langle \bar{g},a\rangle$. 

Our construction yields that $a<a_1<a_2<\cdots$ and $U_a\subsetneq U_1\subsetneq U_2\cdots$; and all $U_i$ are $T$-resistant subgroups of $\G\langle a\rangle$.

By Lemma \ref{2.13}, each $(U_n\cap{\mathcal P}\langle \bar{g},a\rangle)$ must also be a $T$-resistant subgroup $\G_n$ of ${\mathcal P}\langle \bar{g},a\rangle$, and $\G_n
\precneq
\G_{n+1}$, with $\G_n\subseteq U_n\cap {\mathcal P}\langle \bar{g},a\rangle$ and $a_{n+1}\in G_{n+1}\setminus G_n$. This chain is impossible, since ${\mathcal P}\langle \bar{g},a\rangle$ was assumed to be prime and thus have finite rank (where rank here refers to the number of generators of ${\mathcal P}\langle \bar{g},a\rangle$ over the prime model ${\mathcal P}\langle \emptyset\rangle$). This concludes the proof of Lemma \ref{mainlemma}.

\begin{cor}\label{maincor} With $T$, $\G$, $U$, and $a$ as in the statement of the main lemma, the set 
$$U_* := \{x\in \G\langle a\rangle : |x|\leq u\text{ for some }u\in U\}$$ 
is the unique $T$-resistant subgroup of $\G\langle a\rangle$ such that $(\G,U)\subseteq (\G\langle a\rangle, U_*)$ and $U_*$ does not contain $a$. 
\end{cor}

With this in mind, given a $T$-resistant subgroup $U$ of $\G\preceq {\mathcal H}$ and an element $a\in {\mathcal H}$ with $|U|<a<|G\setminus H|$, there are two $T$-resistant subgroups $W$ of $\G\langle a \rangle$, such that $(\G,U)\preceq (\G\langle a\rangle, W)$: one for which $a\in W$ (by Lemma \ref{mainlemma}), and one for which $a\not\in W$ (by Corollary \ref{maincor}). 

For the proof of Theorem \ref{resQE}, it suffices to show, using the Robinson-Shoenfield test for quantifier elimination, that:
\begin{enumerate}
\item For any substructure $(\G,U)$ of a model of $T^*$, there is a $T^*$-closure $(\widetilde{\G},\widetilde{U})$, such that:
\begin{itemize}
\item $(\G,U)\subseteq (\widetilde{\G},\widetilde{U})\models T^*$, and
\item $(\widetilde{G},\widetilde{U})$ can be embedded over $(\G,U)$ into every model of $T^*$ extending $(\G,U)$. 
\end{itemize}
\item If $(\G,U)\subseteq (\G_1,U_1)$ are models of $T^*$ with $\G\neq \G_1$, there is $a\in G_1\setminus G$ such that $(\G\langle a\rangle,U_1\cap G\langle a\rangle) $ can be embedded over $(\G,U)$ into some elementary extension of $(\G,U)$.
\end{enumerate}

To prove (1), suppose that $(\G,U)$ is a submodel of a model $(\G^*,U^*)$ of $T^*$. Then $\G\subseteq \G^*\models T$. Since $T$ is universally axiomatizable, this implies that $\G\models T$. Additionally, $U\subseteq U^*$, which is a $T$-resistant subgroup of $\G$. If $U^*$ is closed under all $0$-definable continuous functions $f$, then so is $U$. Thus, there are two cases to consider:

\textit{Case 1:} $U\neq G$. In this case, $(\G,U)$ is $T$-resistant, and is its own $T^*$-closure.

\textit{Case 2:} $U=G$. In this case, take an elementary extension ${\mathcal H}\succeq\G$ with a new infinite element $a\in H$ and $G<a$. Define $W$ to be the convex hull of $\G$ in $\G\langle a\rangle $. Then $(\G\langle a\rangle,W)$ is a $T^*$-closure of $(\G,U)$. 

For the proof of (2), let $(\G,U) \subseteq (\G_1,U_1)$ be $T$-resistant  with $\G\neq \G_1$. We proceed by finding a value $a\in G_1\setminus G$ such that $(\G\langle a\rangle,U_1\cap G\langle a\rangle )$ is embeddable over $(\G,U)$ into an elementary extension of $(\G,U)$. 

\textit{Case 1:} $\G_1$ contains an element $a$ with $|U|<a<|G\setminus U|$. 

If $a\in U_1$, then choose an elementary extension $(\G^*,U^*)$ of $(\G,U)$, and an element $b\in U^*$ such that $|U|<b<|G\setminus U|$. Then $a$ and $b$ have the same type over $\G$, so there is a $\G$-isomorphism $h:\G\langle a\rangle\to \G\langle b\rangle\subseteq \G^*$ taking $a$ to $b$. By the Lemma \ref{mainlemma}, $h$ maps $U_1\cap G\langle a\rangle$ to $U^*\cap G\langle b\rangle$, thus $a$ and $h$ satisfy condition (2) of the Robinson-Shoenfield test.

If $a\not\in U_1$, let $b$ be in an elementary extension $(\G^*,U^*)$ of $(\G,U)$ such that $|U|<b<|G\setminus U|$, and such that $b\not\in U^*$. Then $b$, together with $h$, defined as above, satisfy condition (2).

\textit{Case 2:} $\G_1$ does not contain $a$ such that $|U|<a<| G\setminus U|$. If this is the case, then $U_1$ is the convex hull of $U$ in $\G_1$. Now, by the model-completeness of $T$, $\G_1$ is an elementary extension of $\G$ and there is an elementary extension $(\G^*,U^*)$ of $(\G,U)$ such that $\G_1$ is embedded in $\G^*$ over $\G$ by some map $i:\G_1\to \G^*$.  In this case, we have that $U^*$ is the convex hull of $U$ in $i(\G_1)$; thus, $i(U_1)=U^*\cap i(\G_1)$. Thus $i$ embeds $(\G_1,U_1)$ into $(\G^*,U^*)$ over $(\G,U)$. 
\end{proof}

\begin{cor}\label{resSK} Let $\M\models T$ be an o-minimal $L$-structure, with $U\subseteq M$ such that $(\M,U)$ is $T$-resistant.  Given $c\in M\setminus U^{\M}$, the $L(U,c)$-theory of $(\M,U,c)$ has definable Skolem functions.
\end{cor}

\begin{proof}  As $\M$ is an o-minimal expansion of a group, $Th(\M)$ has definable Skolem functions.  Thus, by Lemma~\ref{vantest}, $Th(\M)$ has a definitional expansion
$T^{df}$ that is universally axiomatizable.  Now apply Theorem~\ref{resQE} with respect to $T^{df}$ to obtain the universally axiomatizable, quantifier eliminable $T^*$.
Thus, by Lemma~\ref{vantest} again, $T^*$ has definable Skolem functions.  As having definable Skolem functions is invariant under definitional expansions, 
the $L(U,c)$-theory of $(M,U,c)$ also has definable Skolem functions.
\end{proof}

It is clear that, in the context of an o-minimal structure $\M$ expanding a group with a new convex subset $U$, if $\M'=(\M,U)$ is $T$-resistant, then $\M'$ is valuational; as such Corollary \ref{resSK} is a partial converse to Theorem \ref{nonvalnosk}. The remainder of this subsection is devoted to proving Theorem \ref{valeqconv}, which extends this converse to show that a structure which fails to be $T$-resistant in a particular way must in fact be nonvaluational.

\begin{defn}
For an o-minimal structure $\M$ and a convex $U\subset M$, we say that $U$ forms an \textit{infinite cut} of  if $\text{dcl}_\M(\emptyset)\subseteq U$.
\end{defn}

Note that if $\M$ does not define a positive element, as may be the case if the language of $\M$ does not contain multiplication, it is possible to have an infinite cut whose supremum is a finite real number.

\begin{thm}\label{valeqconv} Let $\M$ be an o-minimal expansion of a group with a definable positive element, and let $U\subset M$ be a properly convex symmetric subset such that the expanded structure $(\M,U)$ forms an infinite cut. If $(\M,U)$ is not $T$-resistant, then $(\M,U)$ is nonvaluational.
\end{thm}

\begin{proof} We begin by showing that if there is some definable function which contradicts $T$-resistance, then there is a counterexample of a special form.

\begin{claim}Suppose that $U$ forms an infinite cut of $\M$, and $(\M,U)$ is not $T$-resistant. Then there is $H:M\rightarrow M$ such that
 \begin{itemize}
 \item $H$ is 0-definable in $\M$, continuous, total, and strictly increasing, and
 \item There are $\alpha,\beta\in U$ such that $H(\alpha)\in U$ and $H(\beta)>U$.
 \end{itemize}
\end{claim}

\begin{proof}[Proof of Claim] Suppose $(\M,U)$ is not $T$-resistant, witnessed by the $0$-definable total continuous function $G:M\rightarrow M$, such that $U$ is not closed under $G$.
Note the condition that $U$ be an infinite cut is necessary: if there a 0-definable element $c\in M\setminus U$, then the constant function $G(x)=c$ is a counterexample to $T$-resistance.
However, because we require that $U$ is an infinite cut, there is no definable element of $M\setminus U$; thus the range of $G$ has a nonempty intersection with $U$ (witnessed, at a minimum, by $G(0)\in U$).

By monotonicity in $\M$, there are disjoint open intervals $I_0,I_1,\ldots ,I_n$, such that $I_i <I_{i+1}$ for each $i$, $M\setminus \bigcup_{i=1}^n I_i$ is finite, and $G\upharpoonright I_i$ is strictly monotone for each $i$. Because $\M$ is o-minimal, $\inf I_i$ and $\sup I_i$ must lie inside $U$, as must $G(\inf I_i)$ and $G(\sup I_i)$. Thus we may assume that the image under $G$ of at least one of $I_0$ or $I_n$ has nonempty intersection with $M\setminus U$. We assume without loss of generality that $G(I_n)\cap (M\setminus U)$ is nonempty (if not, consider $-G(-x)$); further we may assume that $G(I_n)\cap (M\setminus U) > U$ (else we consider $-G(x)$). $G\upharpoonright I_n$ cannot be constant, so it must be strictly increasing. Choose $i$ largest so that $G\upharpoonright I_i$ is strictly decreasing or constant, and $G\upharpoonright I_{i+1}$ is strictly increasing, and define $G^\ast$ as follows:
$$G^*(x)=\left\{\begin{array}{ll}
G(x) & x< I_i\\
-G(x) + 2 G(\inf I_i) & x\in I_i\\
G(x) + 2G(\inf I_i)-2G(\inf I_{i+1})& x>I_i
\end{array}
\right.$$
Proceed reductively on the next largest $i$ at which the function changes behavior. Since $\inf I_i=\sup I_{i+1}$, this process preserves continuity, and at each stage, the modified function is strictly increasing from $I_i$ to positive infinity. After repeating the adjustment finitely many times, the resultant function $G^*$ is strictly increasing on all of $\M$ except for possibly $I_0$, on which $G$ may take a constant value from $U$. If this is the case, define $G^{**}$ as
$$G^{**}(x)=\left\{\begin{array}{ll}
-G^*(x) + 2G(\sup I_0) & x\in I_0\\
G^*(x) & \text{otherwise}
\end{array}
\right.$$
Then $G^{**}(x)$ satisfies the requirements for $H$ stated in the claim.
\end{proof}

Now suppose $H$ is 0-definable in $\M$, continuous, total, and strictly increasing, and let $\alpha, \beta\in U$ with $H(\alpha)\in U$ and $H(\beta)>U$. By Proposition \ref{prop.pluslike}, it suffices to find a definable pluslike function $F:M^2\longrightarrow M$ such that the cut defined by the downward closure of $U$ is \textit{not} $F$-valuational.
Define
$$F(x,y):=H(x+y)$$
That $F$ is pluslike is immediate. To see that the cut is not $F$-valuational, choose any $\varepsilon>0$. Then $\beta\in U$, but we have $F(\beta, \varepsilon)>H(\beta)>U$.
\end{proof}

We are finally ready to prove the difficult direction of Theorem~\ref{big}.

\begin{thm}  \label{big2}
Suppose that $\M$ is an o-minimal group and  $\M'=(\M,U)$ is an expansion by a unary predicate that is interpreted as a downward closed, valuational  $C\subseteq M$.
Then  $Th(\M')$ has definable Skolem functions.
\end{thm}

\begin{proof}  First, by compactness there is always an elementary extension $\M^*=(M^*,U^*)$ of $\M'$ where $M^*$ has a positive element 
$e^*>\text{dcl}_{\M^*}(\emptyset)=\text{dcl}_{\M'}(\emptyset)$.
As $\M^*$ and $\M'$ have the same theory, we may assume that $\M^*=\M'$ and hence $e^*\in M$.  
Next, as `having definable Skolem functions' is invariant under definable expansions, we can freely adjust $U$.
In that vein, we may freely assume that $0\in U$.  If this were not the case, then consider 
$$U'=\{a\in M:-a\not\in U\}$$
Then as $(M,U)$ and $(M,U')$ are bi-definable, we could choose $(M,U')$ in place of $(M,U)$.
So, assume that $0\in U$.  
Next, by replacing $U$ by $U+e^*$ if needed, then we may assume that $U$ describes an infinite cut.
Finally, we can make $U$ symmetric by replacing $U$ by 
$$V=\{a\in M:\hbox{both $a,-a\in U$}\}$$
Thus, we have reduced to the case where we can apply Theorem~\ref{valeqconv} to $(M,U)$ to conclude that $(M,U)$ is $T$-resistant.
So, by Theorem~\ref{resQE}, $Th(M,U,c)$ has definable Skolem functions, where the interpretation of $c$ is any element $a\in M$ with $a>U$.
But, as Definition~\ref{withparam} allows formulas $F_\varphi$ with parameters from the model, it follows that
$Th(M,U)$ also has definable Skolem functions.
\end{proof}

\section{Failure of definable choice}  \label{faildc}
In this brief section we complete the proof of Theorem~\ref{big} by proving that no expansion $(\M,U)$ of an o-minimal group by a downward closed subset of $M^1$ can satisfy definable choice, unless of course,
the interpretation $U^{\M'}$ were already $\M$-definable.

\begin{thm}  Suppose that $\M$ is an o-minimal group and $C\subseteq M$ is downward closed and not $\M$-definable.  Then the  weakly o-minimal expansion
$\M'=(\M,U)$ where $U^{\M'}$ is interpreted as $C$ does not satisfy definable choice.
\end{thm}

\begin{proof}  As $C$ is not $\M$-definable, $\langle C,(M\setminus C)\rangle$ is an irrational cut of $M$.  We split into cases depending on whether or not the cut
is valuational or nonvaluational.  If the cut is nonvaluational, then by Theorem~\ref{nonvalnosk} $Th(\M')$ does not have definable Skolem functions, hence it cannot satisfy definable
choice.  Thus, we may assume that the cut is valuational.  But then, directly from the definition of a valuational cut, the infinitesimal set
$$I=\{\varepsilon\in M:\varepsilon +C=C\}$$
is a proper, 0-definable convex subgroup of $M$.  This subgroup $I$ induces a 0-definable  equivalence relation $\sim_I$ on $M$, defined by
$a\sim_I b$ if and only if $a-b\in I$.  Let $\varphi(x,y)$ denote the 0-definable formula $x-y\in I$, and assume by way of contradiction that there were a definable $F$ for 
$\varphi(x,y)$ as in Definition~\ref{defch}.  But then, as the $\sim_I$-classes are convex, the set $R=$range$(F)$ would be an infinite, discrete subset of $M^1$, directly contradicting the fact that $\M'$ is weakly o-minimal.
\end{proof}

\section{Future directions} There are several opportunities for incremental advancement in our understanding of Skolem functions for models $(\M,U)$, where $\M$ is o-minimal. First, in the valuational case, is to find an explicit algorithm for calculating Skolem functions corresponding to a formula. In \cite{Sha}, Shaw gave an explicit algorithm for calculating Skolem functions in the case of a model $(\M,U)$ that satisfies a condition called $T$-immunity. However, $T$-immunity is strictly stronger than $T$-resistance, and in particular is not satisfied if multiplication is definable in $\M$. A second question is whether the counterexample in \S\ref{skolemfail} is essentially the only obstacle to definable Skolem functions: 
Given the nonvaluational structure $(\M,U)$, could we add a unary function $f$ to the language which satisfies $x<f(x)\land U(x)$ whenever $U(x)$ holds, such that the resulting structure $(\M,U,f)$ has definable Skolem functions?

\bibliographystyle{asl}
\bibliography{dcv93r}
\end{document}